\documentclass{article}[12pt]
\usepackage[dvips]{graphicx}
\usepackage{amsfonts}
\usepackage{amsmath}
\usepackage{amssymb}
\usepackage{amscd}
\usepackage{color}
\usepackage{amsthm}
\usepackage{indentfirst}
\usepackage[hmargin=3cm,vmargin=3cm]{geometry}
\usepackage{parskip}

\newtheorem{thm}{Theorem}
\newtheorem{thm*}{Theorem}
\newtheorem{prop}{Proposition}[section]
\newtheorem{lma}{Lemma}

\newtheorem{cor}{Corollary}

\theoremstyle{definition}

\theoremstyle{remark}

\newtheorem{rmk}{Remark}[subsection]

\newtheorem{qtn}{Question}

\newcommand{\R}{{\mathbb{R}}}
\newcommand{\Z}{{\mathbb{Z}}}

\newcommand{\D}{{\mathbb{D}}}

\newcommand{\del}{\partial}

\newcommand\vol{\operatorname{vol}}

\newcommand{\G}{\mathcal{G}}

\newcommand{\til}[1]{\widetilde{#1}}

\newcommand{\cG}{\mathcal{G}}

\newcommand{\ol}[1]{\overline{#1}}

\def\cl#1{{\mathcal{#1}}}

\DeclareMathOperator{\Diff}{\mathrm{Diff}}
\DeclareMathOperator{\Ham}{\mathrm{Ham}}

\DeclareMathOperator{\Cal}{\mathrm{Cal}}

\DeclareMathOperator{\id}{\mathbf{1}}

\begin{document}

\title{\textbf{The Schwarz-Milnor lemma for braids and area-preserving diffeomorphisms}}


\author{\textsc{Michael Brandenbursky, Micha\l{} Marcinkowski, and Egor Shelukhin}\\}

\date{}
\maketitle

\begin{abstract}
We prove a number of new results on the large-scale geometry of the $L^p$-metrics on the group of area-preserving diffeomorphisms of each orientable surface. Our proofs use in a key way the Fulton-MacPherson type compactification of the configuration space of $n$ points on the surface due to Axelrod-Singer and Kontsevich. This allows us to apply the Schwarz-Milnor lemma to configuration spaces, a natural approach which we carry out successfully for the first time. As sample results, we prove that all right-angled Artin groups admit quasi-isometric embeddings into the group of area-preserving diffeomorphisms endowed with the $L^p$-metric, and that all Gambaudo-Ghys quasi-morphisms on this metric group coming from the braid group on $n$ strands are Lipschitz. This was conjectured to hold, yet proven only for low values of $n$ and the genus $g$ of the surface.
\end{abstract}



\tableofcontents

\section{Introduction and main results}

\subsection{Introduction}


The $L^2$-length of a path of volume-preserving diffeomorphisms, which describes a time-dependent flow of an ideal incompressible fluid, corresponds to the hydrodynamic action of the flow in the same way as the length of a path in a Riemannian manifold corresponds to its energy (cf. \cite{ShnirelmanGeneralized}). Indeed, it is the length of this path with respect to the formal right-invariant Riemannian metric on the group $\G$ of volume preserving diffeomorphisms introduced by Arnol'd in \cite{ArnoldGeodesics}. The $L^1$-length of the same path has a dynamical interpretation as the average length of a trajectory of a point under the flow.

Therefore, following the principle of least action, it makes sense to consider the infimum of the lengths of paths connecting two fixed volume-preserving diffeomorphisms. This gives rise to a right-invariant distance function (metric) on $\G.$ Taking the identity transformation as the initial point, Arnol'd observes that a path whose $L^2$-length is minimal (and equal to the distance) necessarily solves the Euler equation of an ideal incompressible fluid.

It follows from works of Ebin and Marsden \cite{EbinMarsden} that for diffeomorphisms in $\G$ that are $C^2$-close to the identity, the infimum is indeed achieved. Further, more global results on the corresponding Riemannian exponential map were obtained in \cite{EbinPrestonMisiolek},\cite{ShnirelmanRuled} (see also \cite{EbinSymp}). In \cite{ShnirelmanGeometry,ShnirelmanGeneralized} Shnirel'man showed, among a number of surprising facts about this subject, that in the case of the ball of dimension $3,$ the diameter of the $L^2$-metric is bounded. This result is conjectured to hold for all compact simply connected manifolds of dimension $3$ or larger (see \cite{EliashbergRatiu, KhesinWendt, ArnoldKhesin}), while its analogue in the non-simply-connected case is false \cite{EliashbergRatiu,BrandenburskyLpMetrics}. Furthermore, Shnirel'man has conjectured that for compact manifolds of dimension $2,$ the $L^2$-diameter is infinite. 

Shnirel'man's conjecture, and its analogues for $L^p$-metrics, with $p \geq 1$ are by now proven. It follows from results of Eliashberg and Ratiu \cite{EliashbergRatiu} that on compact surfaces (possibly with boundary) other than $T^2$ and $S^2,$ Shnirel'man's conjecture holds for all $p \geq 1.$ Their arguments rely on the Calabi homomorphism $\Cal$ \cite{CalabiHomomorphism} from the compactly supported Hamiltonian group $\Ham_c(M,\sigma)$ to the real numbers in the case of a surface $M$ with non-empty boundary ($\sigma$ is the area form), and on non-trivial first cohomology combined with trivial center of the fundamental group in the closed case. For the two-torus $T^2$ Shnirel'man's conjecture holds by \cite[Appendix A]{BrandShelS2}. Finally, the case of $S^2$ was settled in \cite{BrandShelS2} by means of differential forms on the configuration space related to the cross-ratio map. In \cite{MichalM-diameter} the second author gave a new uniform proof of Shnirel'man's conjecture for all compact surfaces.

The methods that were used to prove Shnirel'man's conjecture are two-dimensional in nature, and have to do with braiding of trajectories  of time-dependent two-dimensional Hamiltonian flows (in extended phase space). Indeed, Shnirel'man has proposed to use relative rotation numbers to bound from below the $L^2$-lengths of two-dimensional Hamiltonian paths in \cite{ShnirelmanGeneralized}. This direction is related to the method of \cite{EliashbergRatiu} by a theorem of  Fathi \cite{FathiThese} and Gambaudo and Ghys \cite{GambaudoGhysEnlacements} (see also \cite{ESER, Haissinsky,Hum}). This theorem shows that the Calabi homomorphism is proportional to the relative rotation number of the trajectories of two distinct points in the two-disc $\D$ under a Hamiltonian flow, averaged over the configuration space of ordered pairs of distinct points $(x_1,x_2)$ in the two-disc. 

This line of research was notably pursued in \cite{GambaudoLagrange}, and further in \cite{BenaimGambaudo}, \cite{CrispWiest}, \cite{BrandenburskyLpMetrics}, \cite{BrandenburskyKedra1}, \cite{KimKoberdaAntiTrees}, \cite{BrandShelS2} obtaining quasi-isometric and bi-Lipschitz embeddings of various groups (right-angled Artin groups and additive groups of finite-dimensional real vector spaces) into $\Ham_c(\D^2,dx\wedge dy),$ into $\ker(\Cal) \subset \Ham_c(\D^2,dx\wedge dy),$ and into $\Ham(S^2,\sigma)$ endowed with their respective $L^p$-metrics (see \cite{BrandenburskyKedra2} for similar embedding results on manifolds with a sufficiently complicated fundamental group). In all cases, the key technical estimate is an upper bound, via the $L^p$-length of an isotopy of volume-preserving diffeomorphisms, of the average, over all points in a configuration space of the manifold, of the word length in the fundamental group of the trace of the point under the induced isotopy (closed up to a loop by a system of short paths on the configuration space). 

Such estimates were initially produced by means of analyzing relative rotation numbers of pairs of braids, or quadruples as in \cite{BrandShelS2}. However, in the case of a single braid, as observed by Polterovich, a simpler estimate is possible via the Schwarz-Milnor lemma\footnote{``Schwarz" is a transliteration from Russian; alternative spellings ``\v{S}varc", ``Schwartz" sometimes appear in the literature.} \cite{Efremovic,Svarc,Milnor}. Indeed, according to this result, the universal cover of compact proper length space, in this case our surface, is quasi-isometric to its fundamental group. In this way one can control the word length of trajectories in terms of their lengths, which allows one to prove the required upper bound. A similar estimate in the case of braids on more than one strand is not as readily available, because the configuration space $X_n(M)$ of $n$ points on $M$ is not a compact metric space. The arguments of the second author in \cite{MichalM-diameter} rely crucially on a new complete metric on $X_n(M).$ However, with this metric $X_n(M)$ still can not be considered to be a compact metric space from the point of view of the Schwarz-Milnor lemma.

In this paper we show how to successfully carry out the strategy of the Schwarz-Milnor lemma for configuration spaces. While hints of a similar approach can be discerned in \cite{GambaudoPecou} in the special case of the two-disc and of double collisions, it was not known earlier to be applicable in the general context discussed herein. Specifically, we consider a metric on $X_n(M)$ coming from a natural compactification $\overline{X}_n(M)$ thereof, which is a compact geodesic metric space. This compactification is equivariant under the action of the diffeomorphism group $\Diff(M)$ on $X_n(M)$ in the sense that the action extends naturally to the compactification\footnote{It is curious to note that the same is {\em not} true for the action of $\mathrm{Homeo}(M)$ as was recently proven \cite{Kupers-notop}.}. Furthermore, the map $\pi_1(X_n(M)) \to \pi_1(\overline{X}_n(M))$ induced by the inclusion is an isomorphism. Hence we may apply the Schwarz-Milnor lemma to $\overline{X}_n(M).$ This, and further comparison to the metric from \cite{MichalM-diameter} allows us to prove our main estimate.

The compactification we use was introduced by Axelrod-Singer in \cite{AxelrodSinger-comp} and by Kontsevich in \cite{Kontsevich-ECM, Kontsevich-comp}, inspired by the Fulton-MacPherson compactification in algebraic geometry \cite{FultonMacPherson-comp}. The compactification $\ol{X}_n(M)$ is roughly speaking a certain positive oriented blow-up of $M^n$ along its multi-diagonals. For instance, one part of its codimension $1$ boundary stratum is identified with the disjoint union of spaces of the form $N_1(D_{ij}(M))|_{D^0_{ij}(M)}$ where $N_1$ is the unit normal bundle, $D_{ij}(M) \subset M^n$ is the submanifold of points $(x_1,\ldots,x_n)$ where $x_i = x_j$ and $D^0_{ij}(M) \subset D_{ij}(M)$ is the open dense subset where $x_k \neq x_i = x_j$ for all $k \in \{1,\ldots,n\} \setminus \{i,j\}.$ Note that this is an $S^1$-bundle. Intuitively, from a physical perspective, this means that one resolves a double collision of points by recording the collision point and the direction in which they have collided. Other parts of the codimension $1$ stratum correspond to simple $k$-tuple collisions, and correspond to $S^{2k-3}$-bundles normal to $k$-diagonals. Higher strata correspond to more complicated collisions modeled by suitable graphs. It will, however, be technically most convenient for us to use a model of this compactification recently constructed directly as a subspace of a Euclidean space by Sinha \cite{Sinha-comp}. We describe this construction in Section \ref{sec: proofs} below. 





Finally, we observe that lower bounds on the average word length can often be provided by quasimorphisms - functions that are additive with respect to the group multiplication - up to an error which is uniformly bounded (as a function of two variables). The quasimorphisms we consider here were introduced and studied by Gambaudo and Ghys in the beautiful paper \cite{GambaudoGhysCommutators} (see also \cite{2015-IJM, PolterovichDynamicsGroups, PyTorus, PySurfacesQm, BrandenburskyKnots, BrandShelS2, BrandShelAut}). These quasimorphisms essentially appear from invariants of braids traced out by the action of a Hamiltonian path on an ordered $n$-tuple of distinct points in the surface (suitably closed up), averaged over the configuration space $X_n(M)$ of $n$-tuples of distinct points on the surface $M$. As one of our results, we prove that all homogeneous Gambaudo-Ghys quasi-morphisms are Lipschitz in the $L^p$-metric for all $p\geq 1.$ This subsumes all previous results in this direction and provides a maximally general result. It also contrasts a recent result of Khanevsky according to which none of these quasi-morphisms are continuous in Hofer's metric \cite{K-cont}.

We prove further stronger results on the large-scale geometry on the $L^p$-metric on $\cG=\Diff_{c,0}(M,\sigma)$ for a compact surface $(M,\sigma)$ with an area form. In particular, we provide bi-Lipschitz group monomorphisms of $\R^m$ endowed with the standard (say Euclidean) metric into $(\cG,d_{L^p})$ for each positive integer $m$ and each $p \geq 1.$ Finally, our methods combined with an argument of Kim-Koberda \cite{KimKoberdaAntiTrees} (cf. Crisp-Wiest \cite{CrispWiest}, Benaim-Gambaudo \cite{BenaimGambaudo}) show the existence of quasi-isometric group monomorphisms from each right-angled Artin group to $(\cG,d_{L^p})$ for each $p \geq 1.$ We note that this was previously known only for $\mathbb{D}^2$ and $S^2$ \cite{KimKoberdaAntiTrees,BrandShelS2}. 

Let $M_1 \hookrightarrow M_2$ be a measure preserving embedding of surfaces. 
It is an open question if the induced Lipschitz monomorphism $\Diff_0(M_1,\sigma_1) \hookrightarrow \Diff_0(M_2,\sigma_2)$ is a quasi-isometric embedding. Note that our results provide a partial positive answer to this problem, see Remark \ref{rmk: question embedding}.

\subsection*{Acknowledgments.}
MB~was partially supported by Humboldt research fellowship.
MM~was partially supported by NCN (Sonatina 2018/28/C/ST1/00542) and by the GA\v{C}R project 19-05271Y and by RVO: 67985840.
ES~was partially supported by an NSERC Discovery Grant, by the Fondation Courtois, and by a Sloan Research Fellowship.

\subsection{Preliminaries}

\subsubsection{The $L^p$-metric}

Let $M$ denote a smooth oriented manifold without boundary that is either closed, or $M = X \setminus \del X$ for a compact manifold $X.$ Let $M$ be endowed with a Riemannian metric $g$ and smooth measure $\mu$ (given by an orientation on $M$ and volume form, which in our case of a surface $M$ is an area form $\sigma$). We require that $g$ and $\mu$ extend continuously to $X$ in the second case. Finally denote by \[\G=\Diff_{c,0}(M,\mu)\] the identity component of the group of compactly supported diffeomorphisms of $M$ preserving the smooth measure $\mu.$ In other words, if $M = X \setminus \del X$, it is the identity component of the group of measure preserving diffeomorphisms of $X$ fixing point-wise a neighbourhood of $\del X$. 

Fix $p\geq 1.$ For a smooth isotopy $\{\phi_t\}_{t \in [0,1]},$ from $\phi_0 = \mathrm{1}$ to $\phi_1 = \phi,$  we define the $L^p$-length by \[{l}_p(\{\phi_t\}) = \int_0^1  \left(\frac{1}{\vol(M,\mu)} \cdot \int_M |X_t|^p d\mu\right)^{\frac{1}{p}} \, dt \,,\] where $X_t = \frac{d}{dt'}|_{t'=t} \phi_{t'} \circ \phi_t^{-1}$ is the time-dependent vector field generating the isotopy $\{\phi_t\}$, and $|X_t|$ is its length with respect to the Riemannian structure on $M$. As is easily seen by a displacement argument, the $L^p$-length functional determines a non-degenerate norm on $\G$ by the formula \[d_p(\id,\phi) = \inf \; l_p(\{\phi_t \}).\] This in turn defines a right-invariant metric on $\G$ by the formula \[d_p(\phi_0,\phi_1) = d_p(\id,\phi_1 {\phi_0}^{-1}).\]

\begin{rmk}
Consider the case $p = 1.$ It is easy to see that the $L^1$-length of an isotopy is equal to the average Riemannian length of the trajectory $\{\phi_t (x)\}_{t \in [0,1]}$ (over $x \in M,$ with respect to $\mu$). Moreover for each $p \geq 1,$ by Jensen's (or H\"{o}lder's) inequality, we have \[l_p(\{\phi_t\}) \geq l_1(\{\phi_t\}).\]
\end{rmk}

Denote by $\til{\id}$ the identity element of the universal cover $\til{\G}$ of $\G.$ Similarly one has the $L^p$-pseudo-norm (that induces the right-invariant $L^p$-pseudo-metric) on $\til{\G}$, defined for $\til{\phi} \in \til{\G}$ as
\[d_p(\til{\id},\til{\phi}) = \inf \; l_p(\{\phi_t\}),\] where the infimum is taken over all paths $\{\phi_t\}$ in the class of $\til{\phi}.$ Clearly $d_p(\id,\phi) = \inf d_p(\til{\id},\til{\phi}),$ where the infimum runs over all $\til{\phi} \in \til{\G}$ that map to $\phi$ under the natural epimorphism $\til{\G} \to \G.$

Up to bi-Lipschitz equivalence of metrics ($d$ and $d'$ are equivalent if $\frac{1}{C}d \leq d' \leq C d$ for a certain constant $C>0$) the $L^p$-metric on $\G$ (and its pseudo-metric analogue on $\til{\G}$) is independent of the choice Riemannian structure and of the volume form $\mu$ on $M.$ In particular, the question of boundedness or unboundedness of the $L^p$-metric enjoys the same invariance property.

{\bf Terminology:} For a positive integer $n,$ we use $A,B,C > 0$ as generic notation for positive constants that depend only on $M,\mu,g$ and $n.$

\subsubsection{Quasimorphisms}

Some of our results have to do with the notion of a quasimorphism. Quasimorphisms are a helpful tool for the study of non-abelian groups, especially those that admit few homomorphisms to $\R.$ A quasimorphism $r: G \to \R$ on a group $G$ is a real-valued function that satisfies
\[r(xy) = r(x) + r(y) + b_r(x,y),\]
for a function $b_r:G\times G \to \R$ that is uniformly bounded:
\[\delta(r): = \sup_{G\times G} |b_r| < \infty.\]

A quasimorphism $\overline{r}:G \to \R$ is called \textit{homogeneous} if $\overline{r}(x^k) = k \overline{r}(x)$ for all $x\in G$ and $k \in \Z$. In this case, it is additive on each pair $x,y \in G$ of commuting elements: $r(xy) = r(x) + r(y)$ if $xy = yx.$

For each quasimorphism $r:G\to \R$ there exists a unique homogeneous quasimorphism $\overline{r}$ that differs from $r$ by a bounded function: \[\sup_G |\overline{r} - r| < \infty.\]
It is called the \textit{homogenization} of $r$ and satisfies
\[\overline{r}(x) = \lim_{n \to \infty} \frac{r(x^n)}{n}.\]

Denote by $Q(G)$ the real vector space of homogeneous quasimorphisms on $G.$

For a finitely-generated group $G,$ with finite symmetric generating set $S,$ define the word norm $|\cdot |_S:G \to \Z_{\geq 0}$ by \[|g|_S = \min \{k\,|\, g = s_1 \cdot \ldots \cdot s_k,\,  \forall\, 1 \leq j \leq k, \, s_j \in S\}\] for $g \in G.$ This is a norm on $G,$ and as such it induces a right-invariant metric $d_S: G \times G \to \Z_{\geq 0}$ by $d_S(f,g) = |gf^{-1}|_S.$ This metric is called the word metric. In this setting, any quasimorphism $r :G \to \R$ is controlled by the word norm. Indeed, for all $g \in G,$ \[|r(g)| \leq \left(\delta(r) + \max_{s \in S} |r(s)|\right) \cdot |g|_S.\] When a specific symmetric generating set $S$ for $G$ can be fixed, we will usually denote $|\cdot |_S$ by $|\cdot |_G$. 

We refer to \cite{CalegariScl} for more information about quasimorphisms.

\subsubsection{Configuration spaces and braid groups}
For a manifold $M,$ which in this paper is usually of dimension $2,$ the configuration space $X_n(M) \subset M^n$ of $n$-tuples of points on $M$ is defined as \[X_n(M) = \{(x_1,\ldots,x_n)|\; \displaystyle{x_i \neq x_j,}\;  1 \leq i < j \leq n\}.\] That is \[X_n(M) = M^n \setminus \bigcup_{1 \leq i < j \leq n} D_{ij}\] where for $1 \leq i < j \leq n,$ the partial diagonal $D_{ij} \subset M^n$ is defined as $D_{ij} = \{(x_1,\ldots,x_n)|\; x_i = x_j\}.$ Note that $D_{ij}$ is a submanifold of $M^n$ of codimension $\dim M.$

Finally, if $\dim M = 2$, we define the pure braid group of $M$ as \[P_n(M) = \pi_1(X_n(M)).\] Noting that the symmetric group $S_n$ on $n$ elements acts freely on $X_n(M),$ we form the quotient $C_n(M) = X_n(M)/S_n$ and define the full braid group of $M$ as \[B_n(M) = \pi_1(C_n(M)).\] We note that $P_n(M)$ and $B_n(M)$ enter the exact sequence $1 \to P_n(M) \to B_n(M) \to S_n \to 1.$ In particular $P_n(M)$ is a normal subgroup of $B_n(M)$ of finite index. We refer to \cite{TuraevKassel} for further information about braid groups.

\subsubsection{Short paths and the Gambaudo-Ghys construction}

Let $M$ be a compact oriented surface. Given a real valued quasimorphism $r$ on $P_n(M) = \pi_1(X_n(M),q)$ for a fixed basepoint $q \in X_n(M)$ there is a natural way to construct a real valued quasimorphism on the universal cover $\widetilde{\G}$ of the group $\G=\Diff_{c,0}(M,\sigma)$ of area preserving diffeomorphisms of the surface $M$. We shall see that in the case of $M \neq T^2$ this induces a quasimorphism on $\G$ itself, because the fundamental group of $\G$ is finite. The same is true for $M = T^2$ where we consider the group $\G = \Ham(M,\sigma)$ of Hamiltonian diffeomorphisms instead. This is not a restrictive condition from the viewpoint of large-scale geometry, since by a small modification of \cite[Proposition A.1]{BrandShelS2}, the inclusion $(\Ham(T^2,\sigma),d_{L^p}) \hookrightarrow (\Diff_0(T^2,\sigma),d_{L^p})$ is a quasi-isometry for all $p\geq 1.$

The construction is carried out by the following steps (cf. \cite{GambaudoGhysCommutators,PolterovichDynamicsGroups,BrandenburskyKnots}).

\begin{enumerate}
\item For all $x \in X_n(M) \setminus Z$, with $Z$ a closed negligible subset (e.g. a union of submanifolds of positive codimension) choose a smooth path $\gamma(x):[0,1] \to X_n(M)$ between the basepoint $q\in X_n(M)$ and $x$. Make this choice continuous in $ X_n(M) \setminus Z$. We first choose a system of paths on $M$ itself. Then we consider the induced coordinate-wise paths in $M^n$, and pick $Z$ to ensure that these induced paths actually lie in $X_n(M)$. After choosing the system of paths $\{\gamma(x)\}_{x \in X_n(M)\setminus Z}$ we extend it measurably to all $x \in X_n(M)$ (obviously, no numerical values computed in the paper will depend on this extension). We call the resulting choice a "system of short paths". 

\item Given a path $\{\phi_t \}_{t \in [0,1]}$ in $\G$ starting at $Id$, and a point $x \in X_n(M)$ consider the path $\{\phi_t \cdot x\}$, to which we then catenate the corresponding short paths. That is consider the loop \[\lambda(x,\{\phi_t\}) := \gamma(x) \# \{\phi_t \cdot x\} \# \gamma(y)^{-1}\] in $X_n(M)$ based at $q,$ where $^{-1}$ denotes time reversal. Hence we obtain for each $x \in X_n(M)\setminus Z \cup (\phi_1)^{-1}(Z)$ an element $[\lambda(x,\{\phi_t\})] \in \pi_1(X_n(M),q)$ (or rather for each $x \in X_n(M)$ after the measurable extension in Step 1).

\item Consequently applying the quasimorphism $r:\pi_1(X_n(M),q) \to \R$ we obtain a measurable function $f:X_n(M) \to \R$. 	   Namely $f(x) = r([\lambda(x,\{\phi_t\})])$. The quasimorphism $\Phi$ on $\widetilde{\G}$ is defined by    \[\Phi([\{\phi_t\}]) = \int_{X_n(M)} f \,d \mu^{\otimes n}.\]
     It is immediate to see that this function is well-defined by topological reasons. The quasimorphism property follows by the quasimorphism property of $r$ combined with finiteness of volume. The fact that the function $f$ is absolutely integrable can be shown to hold a-priori by a reduction to the case of the disc. We note, however, that by Tonelli's theorem this fact follows as a by-product of the proof of our main theorem, and therefore requires no additional proof.
\item Of course this quasimorphism can be homogenized, to obtain a homogeneous quasimorphism $\overline{\Phi}$.
\end{enumerate}
\vspace{2mm}
\begin{rmk}
If $M \neq S^2$ and $M \neq T^2$, then $\pi_1(\G) = 0$ and for $M=T^2$ the same is true for $\G = \Ham(M,\sigma).$
In the case $M=S^2,$ by the result of Smale \cite{Smale} $\pi_1(\G) = \Z/2\Z$.
Hence the quasimorphisms descend to quasimorphisms on $\G$, e.g. by minimizing over the two-element fibers of the projection $\widetilde{\G} \to \G$. For $\overline{\Phi}$, the situation is easier since by homogeneity it vanishes on
$\pi_1 (\G) \subset Z(\til{\G})$, and therefore depends only on the image in $\G$ of an element in $\til{\G}$. We keep the same notations for the induced quasimorphisms.
\end{rmk}

\subsection{Main results}

Recall that we work with the group $\cl G = \Diff_{c,0}(M,\sigma)$ for $(M,\sigma)$ a compact oriented surface with an area form. Our main technical result is the following. 

\medskip

\begin{thm}\label{Theorem: average wordlength bound}
	For an isotopy $\overline{\phi} = \{\phi_t\}$ in $\G,$ the average word norm of a trajectory $\lambda(x,\overline{\phi})$ is controlled by the $L^1$-length of $\overline{\phi}:$
	
	\[W(\overline{\phi}) = \int_{X_n(M)\setminus Z \cup (\phi_1)^{-1}(Z)} |[\lambda(x,\overline{\phi})]|_{P_n(M)} \, d\mu^{\otimes n}(x) \leq A \cdot l_1(\overline{\phi}) +B,\] for certain constants $A,B > 0.$
\end{thm}

\medskip

\begin{rmk}
	Note that $W(\overline{\phi})$ depends only on the class $\til{\phi} = [\overline{\phi}] \in \til{\G}$ of $\overline{\phi}$ in the universal cover $\til{\G}$ of $\G$.
	Hence for all compact oriented surfaces other than $S^2$ and $T^2,$ $W(\overline{\phi})$ does not depend on the chosen isotopy $\overline{\phi}$, but only on the diffeomorphism $\phi_1$. This is because $\G$ is simply connected.
\end{rmk}

Theorem \ref{Theorem: average wordlength bound} has a number of consequences concerning the large-scale geometry of the $L^1$-metric on $\G.$ 
Firstly, as any quasimorphism on a finitely generated group is controlled by the word norm, we immediately obtain the following statement.

\medskip

\begin{cor}\label{Corollary: GG Lip}
	The homogenization $\overline{\Phi}$ of each Gambaudo-Ghys quasimorphism $\Phi$ satisfies $$|\overline{\Phi}(\phi)| \leq C \cdot d_1(\phi,1).$$
\end{cor}

In particular, Corollary \ref{Corollary: GG Lip} implies that all Gambaudo-Ghys quasimorphisms are continuous in the $L^1$-metric
(and hence in the $L^p$-metric, see Remark \ref{Remark: Jensen 1 to p}), a fact which was known so far only for the genus zero case (see \cite{BrandShelS2, BrandenburskyLpMetrics}) and for the higher genus case only when one considers 
Gambaudo-Ghys quasimorphisms coming from the fundamental group $P_1(M)$, see \cite{BrandenburskyLpMetrics}. Note that none of these quasimorphisms are continuous in the Hofer metric by a recent result by Khanevsky \cite{K-cont}.

By a theorem of Ishida \cite{Ishida} in the genus zero case and its generalisation to any compact oriented surface 
\cite[Theorem 2.2]{2019-BM1} (see all well Brandenbursky-Kedra-Shelukhin \cite{2018-CCM} in the genus one case, 
and Brandenbursky \cite{2015-IJM} in the higher genus case),
the image of the map $Q(P_n(M)) \xrightarrow{GG} Q(\G)$ is infinite dimensional for $n \geq 4$. 
Thus $Q(\G)$ is an infinite-dimensional vector space. Hence by Corollary \ref{Corollary: GG Lip} we obtain in particular the following.

\medskip

\begin{cor}\label{Corollary: diam}
	The $L^1$-diameter of $\G$ is infinite.
\end{cor}

\medskip

Let $Ent^k$ be the set of products of at most $k$ entropy zero diffeomorphisms.
Theorem 1 in \cite{2019-BM1} and Corollary \ref{Corollary: GG Lip} imply the following. 

\medskip

\begin{cor}
	For each positive integer $k$, the complement in $\G$ of the set $Ent^k$ 
	contains a ball of any arbitrarily large radius in the $L^1$-metric. In particular,
	the set of non-autonomous diffeomorphisms contains a ball of any arbitrarily large radius in the $L^1$-metric.
\end{cor}

\medskip

In what follows, we apply an argument of Kim-Koberda \cite{KimKoberdaAntiTrees} (cf. Benaim-Gambaudo \cite{BenaimGambaudo} and Crisp-Wiest \cite{CrispWiest}), 
and use Theorem \ref{Theorem: average wordlength bound} in order to obtain the following statement, which generalizes an answer 
to a question of Kapovich \cite{KapovichRAAGs} in the case of $S^2$ \cite{BrandShelS2} to arbitrary compact oriented surfaces.  

\medskip

\begin{cor}\label{Corollary: Kapovich}
	The metric group $(\G,d_1)$ admits a quasi-isometric group embedding from each right-angled Artin group endowed with the word metric.  
\end{cor}

\begin{proof}
	Assume first that $M=\Sigma_g$ a closed surface of a positive genus. The case of $S^2$ was already done in \cite{BrandShelS2} 
	and the case of surfaces with boundary will be discussed in the end of the proof. Let $n\in\mathbb{N}$. Let $D\subset\Sigma_g$ be a smoothly embedded open disc and 
	$z_1,\ldots,z_n$ disjoint points in $D$. For each $i$, let $D_i\subset D$ be an embedded open disc centered at $z_i$ such that
	each short geodesic between two points in $D_i$ lies in $D_i$, and $D_i\cap D_j=\emptyset$ for each $i\neq j$.
	We denote by $\Diff(D; D_1,\ldots,D_n)<\G$ the group which consists of all diffeomorphisms in $\G$ which are compactly supported in $D$ and 
	act by identity on each $D_i$. This subgroup is equipped with the $L^1$-metric coming from $\G$.
	
	\begin{lma}\label{lma:above}
		Let $n>4$. We identify $P_n$ with the pure mapping class group of the disc $D$ punctured at $z_1,\ldots,z_n$. 
		Then the inclusion $P_n\to P_n(\Sigma_g)$ is a quasi-isometric embedding.
	\end{lma}
	
	\begin{proof}
		Since $P_n$ is quasi-isometric to $B_n$ and $P_n(\Sigma_g)$ is quasi-isometric to $B_n(\Sigma_g)$,
		it is enough to show that the inclusion $B_n\to B_n(\Sigma_g)$ is a quasi-isometric embedding,
		where $B_n$ is identified with the mapping class group of the disc $D$ punctured at $z_1,\ldots,z_n$.
		By results of Goldberg \cite{Goldberg} (c.f. Birman \cite{Birman}) we have that
		$B_n$ is a subgroup of $B_n(\Sigma_g)$. In addition, the composition $$B_n\xrightarrow{i} B_n(\Sigma_g)\xrightarrow{F} MCG_{g,n}$$	is injective, where $i$ is the inclusion, the map $F:B_n(\Sigma_g)\to MCG_{g,n}$ is the point pushing map and 
		$MCG_{g,n}$ is the mapping class group of $\Sigma_g$ punctured at $z_1,\ldots,z_n$. Moreover, by a result of Hamenstadt
		\cite[Theorem 2]{Hamenstadt} the map $F\circ i: B_n\hookrightarrow MCG_{g,n}$ is a quasi-isometric embedding for $n>4$. 
		{Since $F$ is Lipschitz}, it follows follows that $i$ is a quasi-isometric embedding (see \cite[Lemma 2.1]{BrandenburskyKedra2}). 	\end{proof}		
%
	Let us return to the proof of the corollary. Consider the homomorphism $$H: \Diff(D; D_1,\ldots,D_n)\to P_n<P_n(\Sigma_g),$$
	where $H(\phi)$ is an element in $P_n(\Sigma_g)$ represented by the path $(\phi_t(z_1),\ldots,\phi_t(z_n))$ in the configuration space, and $\{\phi_t\}$ is any isotopy in $\G$ between the identity $\phi_0$ and $\phi:=\phi_1$. 
	
	It follows from the results of Kim-Koberda \cite{KimKoberdaAntiTrees} (c.f. Crisp-Wiest \cite{CrispWiest} and Benaim-Gambaudo \cite{BenaimGambaudo}) that for each RAAG $\Gamma$ there exists $n,$ which can be assumed to satisfy $n>4,$ and an embedding of $\Gamma$ into $\Diff(D; D_1,\ldots,D_n)$ whose composition with $H$ is a quasi-isometric embedding of $\Gamma$ into $P_n$.
	By Lemma \ref{lma:above}, it gives as well a quasi-isometric embedding of $\Gamma$ into $P_n(\Sigma_g)$.
	Then in order to obtain a quasi-isometric embedding of $\Gamma$ into $\G$ it is enough to show that the map $H$ is large-scale Lipschitz whenever $n>4$. 
	
	This fact follows from our main result, Theorem \ref{Theorem: average wordlength bound}. Let us prove it.
	Set 
	$$
	X(D_1,\ldots,D_n):=\{(x_1,\ldots,x_n)\in X_n(\Sigma_g)|\thinspace x_i\in D_i\}.
	$$
	Let $\{\phi_t\}$ be any isotopy in $\G$ between the identity $\phi_0 = \id$ and $\phi:=\phi_1$. 
	It follows from Theorem \ref{Theorem: average wordlength bound} that there exist positive constants $A$ and $B$ which depend only
	on $n$ such that 
	$$
	\int_{X(D_1,\ldots,D_n)} |[\lambda(x,\phi)]|_{P_n(\Sigma_g)} \, d\mu^{\otimes n}(x) \leq A \cdot l_1(\phi) +B\thinspace .
	$$
	Note that for each $x\in X(D_1,\ldots,D_n)$ we have $[\lambda(x,\phi)]=H(\phi)$. Thus
	$$
	\int_{X(D_1,\ldots,D_n)} |[\lambda(x,\phi)]|_{P_n(\Sigma_g)} \, d\mu^{\otimes n}(x)=\vol(X(D_1,\ldots,D_n))\cdot |H(\phi)|_{P_n(\Sigma_g)},
	$$
	and we obtain the result, i.e., $H$ is large-scale Lipschitz:
	$$
	|H(\phi)|_{P_n(\Sigma_g)}\leq  \frac{A}{\vol(X(D_1,\ldots,D_n))}\cdot l_1(\phi) +\frac{B}{\vol(X(D_1,\ldots,D_n))}\thinspace .$$

Assume now that $M_1$ is a surface with boundary. We can embed $M_1$ in a closed surface $M_2 = \Sigma_g$ such that the embedding is area preserving. This induces a monomorphism $\iota \colon \Diff_{c,0}(M_1,\sigma_1) \hookrightarrow \Diff_{c,0}(M_2,\sigma_2)$ (note that elements of
$\Diff_{c,0}(M_1,\sigma_1)$ fix point-wise the neighbourhood of the boundary, so they can be extended by the identity). 
Since $\iota$ it Lipschitz, and the embedding of $\Gamma$ to $\Diff_{c,0}(M_2,\sigma_2)$ can be arranged to factor through $\Diff_{c,0}(M_1,\sigma_1)$, we get a quasi-isometric embedding to $\Diff_{c,0}(M_1,\sigma_1)$ (see \cite[Lemma 2.1]{BrandenburskyKedra2}). \end{proof}

\begin{rmk}
	We note that Corollary \ref{Corollary: Kapovich} implies Corollary \ref{Corollary: diam}, 
	providing the latter with a proof that does not use quasimorphisms.
\end{rmk}

\medskip

Furthermore, note that Corollary \ref{Corollary: Kapovich} implies that for each $k\in\mathbb{N}$ there is a quasi-isometric
embedding $i_k:\Z^k\to (\G,d_1)$. Moreover, there exists a $k$-tuple of autonomous Hamiltonian 
flows (one-parameter subgroups) $\{\{\phi_i^t\}_{t \in \R}\}_{1 \leq i \leq k}$ which have disjoint supports, and
$i_k(e_j)=\phi_j^1$ for each $1\leq j\leq k$, where $e_j=(0,\ldots,1,\ldots,0)$ and 1 lies in the $j$-th entry. The above $k$-tuple of flows defines a homomorphism $j_k:\R^k \to (\G,d_1)$ such that $i_k = j_k|_{\Z^k}.$ Since $\R^k$ is quasi-isometric to $\Z^k$, we obtain the following statement:

\medskip

\begin{cor}\label{Corollary: vector spaces 1}
	The metric group $(\G,d_1)$ admits a quasi-isometric embedding from $(\R^k,d)$ where $d$ is any metric on $\R^k$ induced by a vector-space norm.
\end{cor}

\medskip

The following lemma allows us to prove that the quasi-isometric embeddings of $\Z^k,\R^k$ from Corollary \ref{Corollary: vector spaces 1} are in fact bi-Lipschitz embeddings.

\medskip

\begin{lma}\label{lma: qi to bi-Lip}
Let $(G,d)$ be a metric group and let $A \subset V$ be a subgroup of a normed linear space. Then each Lipschitz quasi-isometric embedding $j:A \to G$ of metric groups is a bi-Lipschitz embedding. If $A$ is discrete and finitely generated then each homomorphism $j:A \to G$ is Lipschitz.
\end{lma}

\begin{proof}
Let $j: A \to G$ be a homomorphism. If $A$ is discrete and finitely generated, then $A \cong \Z^l$ for $l \in \Z_{\geq 0},$ the norm being equivalent to the standard norm on $\Z^l.$ The upper bound $d(j(x),\id) \leq C |x|$ for all $x \in A$ is now immediate. Hence $j$ is Lipschitz. 

Suppose now that $A \subset V$ is a subgroup and $j:A \to G$ is a Lipschitz quasi-isometric embedding. Let $C \geq 1$ be such that \begin{equation}\label{eq: lower qi bound} \frac{1}{C} |x| - B \leq d(j(x),\id) \leq C |x| \end{equation} for a constant $B \geq 0$ and all $x \in A.$ We claim that the inequality \eqref{eq: lower qi bound} holds with $B=0.$ Indeed, consider \eqref{eq: lower qi bound} for $x^m$ where $m \in \Z_{>0}.$ We have \[ \frac{1}{C} m |x| - B \leq d(j(x^m),\id) = d(j(x)^m,\id) \leq m\cdot d(j(x),\id),\] where the last inequality is due to the right-invariance of $d.$ Dividing by $m$ yields \[ \frac{1}{C} |x| - B/m \leq d(j(x),\id), \] which finishes the proof by taking limits as $m \to \infty.$
\end{proof}

\medskip

This immediately implies the following strengthening of Corollary \ref{Corollary: vector spaces 1}, since $j_k$ is evidently Lipschitz.

\medskip

\begin{cor}\label{Corollary: vector spaces 2}
	The metric group $(\G,d_1)$ admits a bi-Lipschitz embedding from $(\R^k,d)$ where $d$ is any metric on $\R^k$ induced by a vector-space norm.
\end{cor}

\medskip

\begin{rmk}\label{Remark: bilipschitz-qm}
	It should be possible to prove Corollary \ref{Corollary: vector spaces 2} by using quasimorphisms, as in \cite{BrandShelS2}.
\end{rmk}

\medskip 

\begin{rmk}\label{Remark: Jensen 1 to p}
	Let $p \geq 1.$ Note that since, by Jensen's (or H\"{o}lder's) inequality, \[d_1 \leq d_p,\] all the above results for $d_1$ continue to hold for $d_p.$
\end{rmk}

\medskip 

Recall that if $M$ has boundary then we assume that elements of $\Diff_{c,0}(M,\sigma_1)$ fix point-wise an open neighbourhood of $\partial M$. Thus if $(M_1,\sigma_1) \hookrightarrow (M_2,\sigma_2)$ is a measure preserving embedding of manifolds, then extending a diffeomorphism on $M_1$ by the identity to a diffeomorphim of $M_2$ gives a well defined 
monomorphism $\iota \colon \Diff_{c,0}(M_1,\sigma_1) \hookrightarrow \Diff_{c,0}(M_2,\sigma_2)$.

We finish this section with a question.

\medskip

\begin{qtn} 
Let $(M_1,\sigma_1) \hookrightarrow (M_2,\sigma_2)$ be a measure preserving embedding of surfaces.
Is the monomorphism $\iota \colon \Diff_{c,0}(M_1,\sigma_1) \hookrightarrow \Diff_{c,0}(M_2,\sigma_2)$ a quasi-isometric embedding?
\end{qtn}

\medskip

\begin{rmk}\label{rmk: question embedding}
Note that this question is motivated by our results since the statement holds for the compositions of $\iota$ with the embeddings of right-angled Artin groups and $(\R^k,d)$ from the proofs of  Corollary \ref{Corollary: Kapovich} and Corollary \ref{Corollary: vector spaces 2}.
\end{rmk}

\subsection{Outline of the proof}

In order to show Theorem \ref{Theorem: average wordlength bound}, 
we define a Riemannian metric $g$ on $X_n(M)$ such that $d_g$, the geodesic metric on $X_n(M)$ induced by $g$,
extends to the geodesic metric on the compactification $\ol{X}_n(M)$.
This allows us to use the Schwarz-Milnor lemma.  

More precisely, Theorem \ref{Theorem: average wordlength bound} is a consequence of the following two propositions 
(we defer the proofs to Section \ref{sec: proofs}).




%
%
%
%
%
%
%

\medskip

\begin{prop}\label{Prop: top}
Let $\lambda$ be a piecewise $C^1$ loop in $X_n(M)$ based at $q.$ Let $S$ be a finite generating set of $P_n(M).$  The word norm of the class $[\lambda] \in \pi_1(X_n(M),q) \cong P_n(M)$ with respect to $S$ satisfies \[ |[\lambda]|_{P_n(M)} \leq A_0 \cdot l_g(\lambda) + B_0,\] for constants $A_0,B_0 > 0$ depending only on $S$ and $n.$ 

\end{prop}

%

The next proposition says that the average length of loops $\lambda(x,\overline{\phi})$ is controlled 
by $l_1(\overline{\phi})$.

\medskip

\begin{prop}\label{prop: analysis}
Let $\overline{\phi} = \{\phi_t\}$ be an isotopy in $\G$ such that $\phi_0 = 1$. There exist constants $A_1,B_1 > 0$ depending only on $n,$ such that 
\[\int_{X_n(M)\setminus Z \cup (\phi_1)^{-1}(Z)} l_g(\lambda(x,\overline{\phi}))  \, d\mu^{\otimes n}(x) \leq A_1 \cdot l_1(\overline{\phi}) +B_1.\]
\end{prop}

To show this proposition, we first compare a metric $g$ to an auxiliary metric $g_0$. 
The metric $g_0$ does not extend to a metric on the compactification $\ol{X}_n(M)$, but 
it is relatively easy to show \cite[Lemma 5.2]{MichalM-diameter}, 
that the inequality from Proposition \ref{prop: analysis} holds 
for $g_0$.

\section{Proofs.}\label{sec: proofs}

\subsection{Compactification of the configuration space}\label{subsection: compactification}

In this section $M$ is a compact $m$-dimensional manifold. 
Below we describe the compactification of $X_n(M)$ using an embedding of the configuration space into a high-dimensional 
Euclidean space. We follow closely the construction given in \cite{Sinha-comp}.

Let us start with describing a family of maps on $X_n(\mathbb{R}^d)$. 
Let $\pi_{ij} \colon X_n(\mathbb{R}^d) \to S^{d-1}$ by defined by the formula \[ \pi_{ij}(x) = \frac{x_i-x_j}{|x_i-x_j|}, \]

where $0 \leq i < j \leq n$ and $x = (x_1,\ldots,x_n) \in X_n(\mathbb{R}^d).$
Let $s_{ijk} \colon X_n(\mathbb{R}^d) \to [0,\infty]$ be defined by the formula \[s_{ijk}(x) = \frac{|x_i-x_j|}{|x_i-x_k|},\]

where $0 \leq i<j<k \leq n$ and $[0,\infty]$ is the one point compactification of $[0,\infty)$. 
Let $\iota \colon X_n(\mathbb{R}^d) \hookrightarrow (\mathbb{R}^d)^n$ be the standard inclusion
and let $A_n[\mathbb{R}^d] = (\mathbb{R}^d)^n \times (S^{n-1})^{n \choose 2} \times [0,\infty]^{n \choose 3}$.
We consider the embedding of $X_n(\mathbb{R}^d)$ to the ambient space $A_n[\mathbb{R}^d]$ given by
the product map \[ \alpha_n = \iota \times (\pi_{ij}) \times (s_{ijk}) \colon X_n(\mathbb{R}^d) 
\to A_n[\mathbb{R}^d].\]

Suppose $M$ is a submanifold of $\mathbb{R}^d$. 
Let $\ol{X}_n(M)$ be the closure of the image $\alpha_n(X_n(M))$ in 
$A_n[\mathbb{R}^d]$.
Then $\ol{X}_n(M)$ is a manifold with boundary, and the inclusion $X_n(M) \hookrightarrow \ol{X}_n(M)$ is a 
homotopy equivalence \cite[Theorem 4.4 and Corollary 4.5]{Sinha-comp}.


\subsection{Two metrics on $X_n(M)$ and proof of Proposition \ref{Prop: top}}

On $[0,\infty]$ we introduce the structure of a smooth manifold by a $1$-map atlas $e^{-x} \colon [0,\infty] \to [0,1]$.
Let $g_{exp}$ be the Riemannian metric on $[0,\infty]$ given by the pull-back of the Euclidean metric
from $[0,1]$. 
In particular, we have that $|\del_x|_{g_{exp}} = e^{-x}$, where $\del_x$ is a standard 'unit' vector field on $[0,\infty)$.  

Let $euc$ be the metric on $A_n[\mathbb{R}^d]$ given by the product of standard metrics on $\mathbb{R}^d$, $S^{d-1}$ and 
$([0,\infty],g_{exp})$.
The first metric on $X_n(\mathbb{R}^d)$ we want to consider is defined to be the pull-back of $euc$ to $X_n(\mathbb{R}^d)$ 
by the map $\alpha_n$: \[ g = \alpha_n^*(euc).\]

Now, since we regard $M$ as a submanifold of $\mathbb{R}^d$,
$g$ induces a Riemannian metric on $X_n(M)$.

The second metric is defined as follows.
Let $x = (x_1,\ldots,x_n) \in X_n(\mathbb{R}^d)$. 
Let $d(x)$ denote the minimal distance between the points in $x$, that is: \[ d(x) = \min\{|x_i - x_j|~\colon~i \neq j\},\] where $|x-y|$ denotes the standard Euclidean distance between vectors $x,y \in \mathbb{R}^d$.  

Note that on $X_n(\mathbb{R}^d)$ we have the Euclidean metric restricted from $(\mathbb{R}^d)^n$. 
We rescale this metric by the factor $\frac{1}{d}$, i.e. we define a metric $g_0$ on $X_n(\mathbb{R}^d)$ by \[ |v|_{g_0} = \frac{|v|}{d(x)},\] where $v \in T_x(X_n(\mathbb{R}^d)) = (\mathbb{R}^d)^n$ and $|v|$ is the Euclidean length of $v \in (\mathbb{R}^d)^n$.
One should note that $d(x)$, and consequently $g_0$, is continuous, but not differentiable. 
A manifold with such a metric is called a $C^0$-Riemannian manifold. 
On a $C^0$-Riemannian manifold one defines the lengths of paths and a geodesic metric in the same way as in the smooth case.

Again, $g_0$ restricts to a $C^0$-Riemannian metric on $X_n(M)$. 

\begin{proof}[Proof of Proposition \ref{Prop: top}]
We need to show that 
\[ |[\lambda]|_{P_n(M)} \leq A_0 \cdot l_g(\lambda) + B_0,\] 
where $\lambda$ be a piecewise $C^1$ loop in $X_n(M)$ based at $q$,
$[\lambda] \in \pi_1(X_n(M),q) \cong P_n(M)$ and $|[\lambda]|_{P_n(M)}$ is the word norm of $[\lambda]$. 

By \cite[Theorem 4.4]{Sinha-comp} $\ol{X}_n(M)$ is a manifold with corners and $X_n(M)$ is its interior.
Moreover, the coordinate charts on $\ol{X}_n(M)$ are defined in such a way that the embedding of 
$\ol{X}_n(M)$ into $A_n[\mathbb{R}^d]$ is smooth. 
Thus one can restrict the metric $euc$ to $\ol{X}_n(M)$.
In other words, $g = \alpha_n^*(euc)$ extends to the compactification $\ol{X}_n(M)$.
Now the proposition follows directly from the Schwarz-Milnor lemma \cite{Efremovic,Svarc,Milnor}.
More precisely, consider the following formulation of this result \cite[Proposition 8.19]{BridsonHaefliger-book}.

\begin{lma}\label{Lma:SM}
Let a group $\Gamma$ act properly and discontinuously by isometries on a proper length space $X.$ If the action is cocompact,
then $\Gamma$ is finitely generated and for any choice of base-point $x_0 \in X,$ the map $\Gamma \to X,$ $h \mapsto h \cdot x_0,$ 
is a quasi-isometry.\end{lma} 
Denote by $g_c$ the metric on $\ol{X}_n(M)$ obtained from $euc$. 
We apply Lemma \ref{Lma:SM} to $X = \til{\ol{X}}_n(M),$ $\Gamma = \pi_1({\ol{X}}_n(M),\pi(x_0)) \cong  \pi_1({X}_n(M),\pi(x_0)),$ and the length space structure on $X$ being induced by the lift $\til{g_c}~=~\pi^*g_c$ to $X$ of $g_c$ on $K = X/\Gamma = {\ol{X}}_n(M)$ by the natural projection map $\pi:X \to K.$ Note that $d_{\til{g_c}}(x_0,h \cdot x_0) \leq l_{g_c}(\lambda)$ where $\lambda$ is any $C^1$ loop in $X_n(M)$ based at $\pi(x_0)$ representing the element~$h$. 
Moreover, since $g_c$ is an extension of $g$, we have $l_{g_c}(\lambda) = l_g(\lambda)$. 
\end{proof}

\subsection{Proof of Proposition \ref{prop: analysis}}

We begin with the proof of the following technical result.
\medskip

\begin{lma}\label{lma: two metrics}
There exists $C>0$ depending only on $n$, such that $|v|_g \leq C \cdot |v|_{g_0}$ for every $v \in T(X_n(M))$.
\end{lma}

\begin{proof}
Since $\alpha_n$ is a product map, in order to get a bound on the norm $|v|_g$,
we need to bound the norms of the vectors $D\pi_{ij}(v)$ and $Ds_{ijk}(v)$. 
Let $x = (x_1,\ldots,x_n) \in X_n(M)$ and let $v = (v_1,\ldots,v_n) \in T_x(X_n(M))$. 
There exists $\epsilon > 0$ such that $x+tv \in X_n(\mathbb{R}^d)$ for every $t < \epsilon$.
Note that even though the path $\{x+tv\}_{t\in[0,\epsilon]}$ is not contained in $X_n(M)$, 
it still represents the tangent vector $v \in T_x(X_n(M)) < T_x(X_n(\mathbb{R}^d))$. 
We have \begin{align*}
D\pi_{ij}(v) &= \frac{d}{dt}_{|t=0} \, \pi_{ij}(x+tv)&\\
&= \frac{d}{dt}_{|t=0} \, \frac{x_i-x_j + t(v_i-v_j)}{|x_i-x_j + t(v_i-v_j)|}&\\
&=\frac{v_i-v_j}{|x_i-x_j|} - \frac{(x_i-x_j)\langle v_i-v_j, x_i-x_j \rangle}{|x_i-x_j|^3}.
\end{align*}

By definition $D\pi_{ij}(v)$ is a vector in $\mathbb{R}^d$ tangent to a $(d-1)$-dimensional sphere
at point $\pi_{ij}(x)$. By $|D\pi_{ij}(v)|$ we denote its Euclidean length. 

Applying the inequalities $\langle v_i-v_j, x_i-x_j \rangle \leq |v_i-v_j||x_i-x_j|$ and 
$|v_i-v_j| \leq |v_i| + |v_j|$, we obtain:
\begin{align*}
|D\pi_{ij}(v)| &\leq 2\frac{1}{|x_i-x_j|}(|v_i|+|v_j|)&\\
& \leq \frac{2}{d(x)}(|v_i|+|v_j|)&\\
& \leq \frac{2}{d(x)}(\sum_i |v_i|).
\end{align*}

Similarly, for $s_{ijk}$ we compute (already in the map):
\begin{align*}
De^{-s_{ijk}}(v) &= \frac{d}{dt}_{|t=0} \, e^{-s_{ijk}(x+tv)} &\\
&= \frac{d}{dt}_{|t=0} \, e^{-\frac{|x_i-x_j + t(v_i-v_j)|}{|x_i-x_k + t(v_i-v_k)|}}&\\
&= \Big[\frac{\langle v_i-v_j, x_i-x_j \rangle}{|x_i-x_j||x_i-x_k|} - \frac{|x_i-x_j|\langle x_i-x_k, v_i-v_k \rangle}{|x_i-x_k|^3} \Big]e^{-\frac{|x_i-x_j|}{|x_i-x_k|}}.&
\end{align*}

By definition $|Ds_{ijk}(v)|_{g_{exp}} = |De^{-s_{ijk}}(v)|$. Using similar inequalities as before we get:
\begin{align*}
|Ds_{ijk}(v)|_{g_{exp}} 
\leq \Big[ \frac{|v_i|+|v_j|}{|x_i-x_k|} + \frac{(|v_i|+|v_k|)|x_i-x_j|}{|x_i-x_k|^2}\Big]e^{-\frac{|x_i-x_j|}{|x_i-x_k|}}.
\end{align*}

Applying the inequality $e^{-x} \leq \frac{1}{x}$ to the last term we obtain: 

\begin{align*}
|Ds_{ijk}(v)|_{g_{exp}}&\leq \Big[ \frac{|v_i|+|v_j|}{|x_i-x_k|} + \frac{(|v_i|+|v_k|)|x_i-x_j|}{|x_i-x_k|^2}\Big]\frac{|x_i-x_k|}{|x_i-x_j|}&\\
&= \frac{|v_i|+|v_j|}{|x_i-x_j|} + \frac{|v_i|+|v_k|}{|x_i-x_k|}&\\
&\leq \frac{1}{d(x)}(2|v_i|+|v_j|+|v_k|)&\\
&\leq \frac{2}{d(x)}(\sum_i |v_i|).
\end{align*}

Finally, since $\sum_i |v_i| \leq\ \sqrt{n}|v|$, we get:
\[ |D\pi_{ij}(v)| \leq \frac{\sqrt{n}}{d(x)}|v|, \]
\[ |Ds_{ijk}(v)|_{g_{exp}} \leq \frac{2\sqrt{n}}{d(x)}|v|. \]

We assume that $M$ is compact, therefore there exists $A>0$ such that $d(x) \leq A$ for every $x \in X_n(M)$. 
Now we can bound the pulled-back metric $|\cdot|_g$: 
\begin{align*}
|v|^2_g &= |D\iota(v)|^2 + \sum_{i,j}|D\pi_{ij}(v)|^2 + \sum_{i,j,k}|Ds_{ijk}(v)|_{g_{exp}}^2 &\\
&\leq \frac{A^2}{d(x)^2}|v|^2 + \sum_{i,j}\frac{n}{d(x)^2}|v|^2 + \sum_{i,j,k}\frac{4n}{d(x)^2}|v|^2 &\\
&= C\frac{|v|^2}{d(x)^2} = C \cdot |v|^2_{g_0},
\end{align*}
for $C = A^2 + n {n \choose 2} + 4n {n \choose 3}$.\end{proof}

\begin{rmk}
In the proof of Lemma \ref{lma: two metrics}, the inequality $e^{-x} < \frac{1}{x}$
is used to show the bound on the length of $Ds_{ijk}$.
Indeed, the choice of the function $e^{-x} \colon [0,\infty] \to [0,1]$ to define a smooth structure and the metric on 
$[0,\infty]$ is not completely arbitrary.  
For a different identification, e.g. $ln(\frac{x+e}{x+1}) \colon [0,\infty] \to [0,1]$, we get a different smooth structure 
on $[0,\infty]$ (in the sense, that the identity map is not smooth) and a metric which is not equivalent to $g_{exp}$. 
Then after pull-back by $\alpha_n$ we get a metric on $X_n(M)$ which
is not equivalent to $g$ and for this metric Lemma \ref{lma: two metrics} might not hold. 
\end{rmk}

\begin{proof}[Proof of Proposition \ref{prop: analysis}]

We need to show that \[\int_{X_n(M)\setminus Z \cup (\phi_1)^{-1}(Z)} l_g(\lambda(x,\overline{\phi}))  \, d\mu^{\otimes n}(x) \leq A_1 \cdot l_1(\overline{\phi}) +B_1,\] where $\overline{\phi} = \{\phi_t\}$ is an isotopy in $\G$ such that $\phi_0 = 1$. 

A similar inequality was proven in \cite{MichalM-diameter} for a metric which is equivalent to $g_0$.
Let us first describe this metric.  
By $d_M$ denote the geodesic metric on $M$ induced by the restriction of the standard Riemannian metric on $\mathbb{R}^d$. 
Let $x \in X_n(M)$ and $v \in T_x(X_n(M))$. We define \[ |v|_{g_b} = \frac{|v|}{d_M(x)},\] where $d_M(x) = \min\{ d_M(x_i,x_j)~\colon~i \neq j\}$ and $|v|$ is the Euclidean length of $v$ seen as a vector in $(\mathbb{R}^d)^n$.  
The difference between the function $d$ used to define $g_0$ and $d_M$ is that
in $d$ we measure the distance between points $x_i$ and $x_j$ in the ambient space 
$\mathbb{R}^d$ and in $d_M$ inside $M$. Since $M$ is compact, clearly $d$ and $d_M$ are equivalent 
and consequently $g_b$ and $g_0$ are equivalent. 

It follows from \cite[Lemma 5.2]{MichalM-diameter} that 
\[\int_{X_n(M)\setminus Z \cup (\phi_1)^{-1}(Z)} l_{g_b}(\{\phi_t \cdot x\})  \, d\mu^{\otimes n}(x) \leq C' \cdot l_1(\overline{\phi}).\] 
Recall that $\{\phi_t \cdot x\}$ is a path in $X_n(M)$ ($\phi_t$ acts on $x \in X_n(M)$ component-wise).
Since $g_0$ and $g_b$ are equivalent, this inequality holds as well for $g_0$ (with a possibly different constant). 

Let us now focus on the metric $g$. 
Since the geodesic metric $d_g$ defined by $g$ extends to the compactification of $X_n(M)$, the diameter of $(X_n(M), d_g)$ is finite.
We can choose the system of short paths $\gamma(x)$ such that $l_g(\gamma(x)) \leq D$ 
for some $D>0$ and every $x \in X_n(M)\setminus Z$.

Thus for every $x \in X_n(M) \setminus Z \cup (\phi_1)^{-1}(Z)$ we have 
\[ l_g(\lambda(x,\overline{\phi})) \leq l_g(\gamma(x)) + l_g(\{\phi_t \cdot x\}) + l_g(\gamma(\phi_1(x))) \leq 2D + l_g(\{\phi_t \cdot x\}). \]

Finally, due to Lemma \ref{lma: two metrics}, we have $l_g(\{\phi_t \cdot x\}) \leq C \cdot l_{g_0}(\{\phi_t \cdot x\})$ and the proposition follows.
\end{proof}

\bibliographystyle{amsplain}
\bibliography{LpSurfacesRefs}

\vspace{3mm}

\textsc{Michael Brandenbusrky, Department of Mathematics, Ben Gurion University, Be’er Sheva 84105, Israel}\\
\emph{E-mail address:} \verb"brandens@math.bgu.ac.il"

\vspace{3mm}

\textsc{Micha\l{} Marcinkowski, Institute of Mathematics, Wroc\l{}aw University, pl. Grunwaldzki~2/4,
50-384 Wroc\l{}aw, Poland}\\
\emph{E-mail address:} \verb"marcinkow@math.uni.wroc.pl"

\vspace{3mm}

\textsc{Egor Shelukhin, Department of Mathematics and Statistics,
	University of Montreal, C.P. 6128 Succ.  Centre-Ville Montreal, QC, H3C 3J7, Canada}\\
\emph{E-mail address:} \verb"shelukhin@dms.umontreal.ca"

\end{document}